\documentclass[a4paper, 10pt]{amsart}

\usepackage{defs}
\usepackage{amsaddr}
\usepackage{mathtools}
\usepackage{empheq}

\usepackage{caption}
\usepackage{subcaption}
\usepackage{multicol}
\hypersetup{final=true}

\title[Rate Constrained Maximum Principle]{A discrete-time Pontryagin maximum principle under rate constraints}
\thanks{Siddhartha Ganguly and Souvik Das are supported by the PMRF grant RSPMRF0262, from the Ministry of Human Resource Development, Govt. of India, respectively.\\
We would like to thank Pradyumna Paruchuri and Shruti Kotpalliwar for the helpful discussions and comments.}

\author{Siddhartha Ganguly and Souvik Das and Debasish Chatterjee and  Ravi Banavar} 
\address{
Systems \& Control Engineering\\ IIT Bombay, Powai\\ Mumbai 400076, India\\
\indent URL: \href{https://sites.google.com/view/siddhartha-ganguly}
{\textsf{https://sites.google.com/view/siddhartha-ganguly}}\\
\texttt{Email:sganguly@iitb.ac.in} \\
\indent URL: \href{https://sites.google.com/view/souvikd}
{\textsf{https://sites.google.com/view/souvikd}}\\
\texttt{Email:souvikd@iitb.ac.in}\\
\indent URL: \href{http://www.sc.iitb.ac.in/~chatterjee}
{\textsf{https://www.sc.iitb.ac.in/\textasciitilde chatterjee}}\\
Email: \texttt{dchatter@iitb.ac.in}\\
\indent URL: \href{https://sites.google.com/view/ravibanavar}
{\textsf{https://sites.google.com/view/ravibanavar/}}\\ Email: \texttt{banavar@iitb.ac.in}
}

\date{\DTMnow}

\begin{document}

\maketitle

\begin{abstract}

Limited bandwidth and limited saturation in actuators are practical concerns in control systems. Mathematically, these limitations manifest as constraints being imposed on the control actions, their rates of change, and more generally, the global behavior of their paths. While the problem of actuator saturation has been studied extensively, little attention has been devoted to the problem of actuators having limited bandwidth. While attempts have been made in the direction of incorporating frequency constraints on state-action trajectories before, rate constraints on the control at the design stage have not been studied extensively in the discrete-time regime. This article contributes toward filling this lacuna. In particular, we establish a new discrete-time Pontryagin maximum principle with \emph{rate constraints} being imposed on the control trajectories, and derive first-order necessary conditions for optimality. A brief discussion on the existence of optimal control is included, and numerical examples are provided to illustrate the results.

\end{abstract}


\section{Introduction}\label{sec:intro}

In this article we study a class of discrete-time optimal control problems which include constraints on the rate of the control actions in addition to the conventional state and control constraints. More specifically, constraints of the following forms have been considered here:
    \begin{enumerate}
        \item \label{obj1}state constraints at every time instant,
        \item \label{obj2}constraints on the control at every time instant, and
        \item \label{obj3}constraints on the rate of the control action at every time instant.
    \end{enumerate}
    Constraints of the type (\ref{obj1}) and (\ref{obj2}) above fall within the ambit of safety specifications and actuator saturation respectively. Under the broad umbrella of bandwidth limitations, actuators frequency constraints on the control trajectories have been studied in the discrete-time regime in \cite{ref:PP:DC-19, ref:PP:SK:KP:DC:RB-20, ref:SK:PP:DC-19}, but little attention has been given to the case of input rate constraints (constraint of type (\ref{obj3})). In this article we derive a Pontryagin maximum principle (PMP) based on the formulation proposed in \cite{ref:VGB-75} incorporating the indicated rate constraints.
    
    \par Rate constraints are inherently present in almost every inertial actuator. In the majority of cases, the response of most actuators are complex and nonlinear, and such properties may adversely affect the performance if not considered at the design stage, thereby leading to performance and stability deterioration. Hence these constraints are of great practical importance. Let us consider several concrete examples: (a) Robotic manipulators are driven via torques applied at various links as control input that need to be constrained along with their derivatives to reduce the effect of vibrations due to internal modes. (b) A large electrical grid of networks with several synchronous generators working in tandem in response to control commands need protection from arbitrary variations in the demand and supply to ensure its good health. (c) The abort landing problem \cite{ref:bulir-1991} of an aircraft in the presence of windshear underscores the importance of rate constraints on the control inputs. Aggressive pilot commands, high gain of the flight control system or some irregularity in the flight-system can trigger actuator rate limitations, thus instigating a pilot induced oscillation (PIO) \cite{ref:YY:IVK:DC-11, ref:hess-1997}. (d) Boiler-turbine systems in process industries which are used for tracking purposes of various load commands, are often subjected to actuator magnitude and rate constraints \cite{ref:BR:AKJ-97, ref:CHEN}. In particular, in nuclear reactors one of the major use of a control system is to control the rate of fission by maintaining the temperature of the coolant along the lines of \cite{ref:CHEN}, where the rate of cooling becomes a natural control parameter.
\subsection*{Background}
    Actuator rate constraints for continuous linear time-invariant systems have been studied in \cite{ref:JBM:BDOA-67} and \cite{ref:IEK:FJ-01}. A study on the controllability properties of mechanical systems with rate and amplitude constraints can be found in \cite{ref:VIM:ESP-04}. Global stabilization procedure of linear systems with bounds on the control actions and its successive derivatives upto an arbitrary order has been considered in \cite{ref:JL:AC:YC-16}. In \cite{ref:RF:LP-98} a recursive Lyapunov function based controller for globally stable feedback synthesis with constraints on the inputs and its rates have been introduced. Implementation of rate constraints in the nonlinear programming solver in the context of continuous time optimal control problems have been reported in \cite{ref:YN:ECK-20}.
    
 The PMP in the discrete-time regime for optimal control problems with constraints of the type (\ref{obj1}) and (\ref{obj2}) have been studied extensively; see \cite{ref:VGB-book} for a book-length treatment where the author based the proofs on the so-called ``tent method''. A version of the discrete PMP was derived in \cite{ref:dubovitskii1978discrete} where a different technique known as the Dubovitskii-Milyutin lemma \cite{ref:dubovitskii65} was employed to arrive at the necessary conditions; see \cite{ref:rojas20} for recent applications of the lemma. Over the next three decades, extensions to systems with weaker regularity requirements were developed, including nonsmooth versions of the PMP. More recently, motivated by engineering applications, a discrete-time PMP on matrix Lie groups and on smooth manifolds were established in \cite{ref:KPDCRV-18} and \cite{ ref:kipka2019discrete} respectively, which stimulated the development of discrete-time PMPs for specific engineering applications, including those incorporating frequency constraints \cite{ref:PP:DC-19, ref:SK:PP:DC-19, ref:MPK:DC:RB}. However, \emph{there is no PMP with rate constraints in the literature. 
    The present article is precisely an attempt to bridge this gap, particularly in the discrete-time setting.}
   \subsection*{Our contributions} 
   In the premise of discrete-time optimal control problems for a general class of time-variant nonlinear systems, as a primary contribution of this article, a discrete-time version of the PMP with the three constraints (\ref{obj1}), (\ref{obj2}) and (\ref{obj3}) has been derived. Consequently, we establish first order necessary conditions for optimality. Moreover, a proof of the existence of the optimal solution in the above context under rate constraints and a PMP for the special case of control-affine nonlinear systems has been included.
 \subsection*{Notation}
    We employ the standard notation: $\N \Let \aset{1,2,\ldots}$ denotes the set of positive integers. The vector space $\Rbb^{n}$ is equipped with standard inner product $\inprod{x}{y}\Let x^{\top} y$ for every $x,y \in \Rbb^{n}.$ By $\big(\Rbb^n\big){\dual}$ we mean the dual space of $\Rbb^n$ which is of course isomorphic to the primal vector space $\Rbb^n$ in view of the Riesz representation theorem.

 \section{Problem Formulation}
\subsection{Original problem}\label{sec:original problem} We consider a discrete-time non-autonomous control system given by the recursion 
    \begin{equation}
        \label{eq:system}
        \stt{}{t+1} =  \field \bigl(t,\stt{}{t},\cont{}{t} \bigr) \quad \text{for all $t=\timestamp{0}{T-1}$},
    \end{equation}
    with the following data:
    \begin{enumerate}[label=\textup{(\ref{eq:system}-\alph*)}, leftmargin=*, widest=b, align=left]
		\item \label{eq:system-state}\(\stt{}{t} \in \Rbb^{d}\) is the vector of states at time \(t;\)
		\item \label{eq:system-control}\(\cont{}{t} \in \Rbb^m\) denotes the control action at time \(t;\)
		\item  \label{eq:system-dynamics} for each \(s=\timestamp{0}{T-1}\), the map \(\Rbb^d \times \Rbb^m \ni (\dummyx,\dummyu) \mapsto \field\bigl(s,\dummyx,\dummyu \bigr) \in \Rbb^d\) is continuous, and it describes the dynamics of \eqref{eq:system}.
	\end{enumerate}
    Consider the optimal control problem
\begin{align}
        \label{eq:original problem}
        \begin{aligned}
            &\minimize_{\stt{}{0},(\cont{}{t})_{t=0}^{T-1}}  && \sum_{t=0}^{T-1}\cost\bigl(t,\stt{}{t},\cont{}{t}\bigr) +\cost_F(T,\stt{}{T})\\
            &\hspace{2mm}\sbjto && \begin{cases}
            \text{dynamics\,}\eqref{eq:system},\\
            \stt{}{t} \in \Mbb{t} \hspace{3mm}\text{for \,} t=\timestamp{0}{T},\\
            \cont{}{t} \in \Ubb{t} \hspace{3.6mm}\text{for \,} t=\timestamp{0}{T-1},\\
            \norm{\cont{}{t+1} - \cont{}{t}}  \leqslant \Rcal_t \hspace{3mm}\text{for \,} t=\timestamp{0}{T-2},
            \end{cases}
        \end{aligned}
    \end{align}
with the following data:
		
		
		
		
\begin{itemize}[leftmargin=*, widest=b, align=left]
		\item \label{eq:original problem-time}\(T \in \N\) is a given control horizon;
		
		\item the data \eqref{eq:system-state}-\eqref{eq:system-dynamics} hold;
		
		\item \label{eq:original problem cost} the map \(\Rbb^d \times \Rbb^m \ni (\dummyx,\dummyu) \mapsto \cost\bigl(s,\dummyx,\dummyu\bigr) \in \Rbb\) is a continuous \emph{cost per stage function} at time \(s\) for each \(s = \timestamp{0}{T-1}\), and \(\Rbb^d \ni \xi \mapsto \cost_F(T,\xi)\in \Rbb\) represents the continuous \emph{final stage cost};
		
		\item \label{eq:original problem-Mt} the set  \(\Mbb{t} \subset \Rbb^d\) is the \emph{set of admissible states} with nonempty relative interior at each \(t = \timestamp{0}{T};\) 
		
		\item \label{eq:original problem-Ut}the set \(\Ubb{t} \subset \Rbb^m\) is the \emph{set of admissible actions} with nonempty relative interior at each \(t=\timestamp{0}{T-1};\) 
		
		\item \label{eq:original problem-rate}
		\(\norm{\cont{}{t+1} - \cont{}{t}} \le \Rcal_t\) for each \(t=\timestamp{0}{T-2},\) where \((\Rcal_t)_{t=0}^{T-2}\) is a sequence of pre-specified positive numbers.
			\end{itemize}
The type of constraints considered in the optimal control problem \eqref{eq:original problem} are:
\begin{itemize}
        \item \label{item:cons_state}
        \textit{State Constraints}: We stipulate that the state trajectory \(\bigl(\stt{}{t}\bigr)_{t=0}^{T}\) lies in the tube \(\prod_{t=0}^{T}\Mbb{t} = \Mbb{0} \times \Mbb{1} \times \cdots \times \Mbb{T} \subset \Rbb^{d(T+1)}\).
        \item \label{item:cons_control}\textit{Control Constraints}: The control trajectory \(\bigl(\cont{}{t}\bigr)_{t=0}^{T-1}\) is permitted to take values in the tube \(\prod_{t=0}^{T-1}\Ubb{t} = \Ubb{0} \times \Ubb{1} \times \ldots \times \Ubb{T-1} \subset \Rbb^{mT}\).
        \item \label{item:cons_rate}\textit{Rate Constraints}: Constraints on the \emph{rate of change of the control actions} are captured by
        \(
            \norm{\cont{}{t+1} - \cont{}{t}} \leqslant \Rcal_t \) for all \(t=\timestamp{0}{T-2}\),
        where \(\Rcal_t > 0\) denotes the rate constraint at the time instant \(t\). 
    \end{itemize}

\subsection*{Existence of solutions to \eqref{eq:original problem}}
We provide a result to establish the existence of solutions for the optimal control problem \eqref{eq:original problem}.

\begin{theorem}\label{thrm:existence}
Consider the optimal control problem \eqref{eq:original problem} with its associated data. Define the optimal state-action trajectory pair by \(\Lambda^* \Let \Bigl(\bigl(\stt{\as}{t}\bigr)_{t=0}^{T},\bigl(\cont{\as}{t}\bigr)_{t=0}^{T-1}\Bigr)\). Let \(\mathcal{X}\) be the feasible set corresponding to \eqref{eq:original problem} which is assumed to be nonempty.
Let the following conditions hold:
\begin{enumerate}[label=\textup{(\alph*)}, leftmargin=*, widest=b, align=right]
    \item\label{thrmi1} The state constraint set \(\Mbb{0} \subset \Rbb^d\) and the control constraint sets \(\Ubb{t} \subset \Rbb^m\) for every \(t=0,\ldots,T-1\) are nonempty compact subsets,
    \item\label{thrmi2} the state constraint sets \(\Mbb{t}\) for every \(t = 1,\ldots,T,\) are nonempty closed subsets of \(\Rbb^d\),
    \item\label{thrmi3} the function \(f\bigl(t, \cdot,\cdot \bigr):\Mbb{t}\, \times\, \Ubb{t} \lra \Mbb{t+1}\) is continuous for every \(t = 0, \ldots,T-1\), and
    \item\label{thrmi4} the function \(c \bigl(t,\cdot,\cdot\bigr):\Mbb{t} \times \Ubb{t} \lra \Rbb\) is lower semi-continuous for every \(t = 0, \ldots,T-1,\) and \(\cost_F(T, \cdot):\Mbb{T} \lra \R\) is lower semi-continuous.
\end{enumerate}
Then the problem \eqref{eq:original problem} admits a solution, i.e., there exists optimal trajectories \(\Lambda^*\). 
Alternatively, suppose the following conditions hold:
\begin{enumerate}[label=\textup{(\alph*)},leftmargin=*, widest=b, align=right, start=5]
\item\label{thrmii1} The constrained set \(\Mbb{0}\subset \Rbb^d\) and the control constrained sets \(\Ubb{t}\subset \Rbb^m\) for every \(t=0,\ldots,T-1\) are nonempty closed subsets of \(\Rbb^d\) and \(\Rbb^m\), respectively,
\item \label{thrmii2} the assumptions \eqref{thrmi2}-\eqref{thrmi4} hold,
\item \label{thrmii3} the function \(\cost(0,\cdot,\cdot)\) is weakly coercive, i.e.,
\[
    \begin{aligned}
    \lim_{\|(\dummyx,\dummyu)\| \to +\infty}\cost(0,\dummyx,\dummyu) = +\infty.
    \end{aligned}
\]
\end{enumerate}
Then the problem \eqref{eq:original problem} admits a solution, i.e., there exists optimal trajectories \(\Lambda^{*}\). 
\end{theorem}
{\it \textbf{Proof}}: We start by giving a proof of the first assertion, which involves the assumptions \eqref{thrmi1}-\eqref{thrmi4}. For convenience we define \(\optpair_t \Let \Bigl(\stt{}{0},\bigl(\cont{}{\tau}\bigr)_{\tau=0}^{t-1}\Bigr) \) where \(\optpair_{0} \Let \stt{}{0}\). Let \(\varphi(t; \cdot): \Rbb^{d} \times \Rbb^{mt} \to \Rbb^d\) denote the solution of the recursion \eqref{eq:system} such that for every \(t=0,\ldots,T\) we have 
\begin{equation*}
    \begin{aligned}
     \begin{cases}
     \varphi\bigl(t; \optpair_{t} \bigr)=\stt{}{t}  = f 
    \Bigl( t-1,  f \bigl( \cdots f\bigl(0,x(0),u(0)\bigl),u(1) \bigr),u(2) \bigr),\cdots\bigr),u(t-1) \Bigr),\\
    \varphi(0;\stt{}{0})=\stt{}{0} .
    \end{cases}
    \end{aligned}
\end{equation*}
Since, the function \(f(t,\cdot,\cdot)\) is continuous for every \(t=0,\ldots,T-1\), continuity of \(\varphi \bigl(t;\cdot \bigr)\) follows. This implies that the set \(\bigcap_{t=0}^T\aset[\big]{\optpair_t \suchthat \stt{}{t} \in \Mbb{t}}\), or equivalently the set
\begin{equation*}
   \scal \Let \bigcap_{t=0}^{T}\varphi\bigl(t; \cdot \bigr)^{-1} \bigl( \Mbb{t}\bigr),
\end{equation*}
is closed and nonempty. Let us define \(\wcal \Let \Mbb{0} \times \Ubb{0} \times \cdots \times \Ubb{T-1}.\) Then the feasible set can be written as 
\begin{equation*}
    \mathcal{X} \Let \wcal \cap \scal \cap \RCal,
\end{equation*}
where \(\RCal \Let \bigcap_{t=0}^{T}\aset[\big]{ \optpair_t \suchthat \stt{}{0} \in  \Mbb{0},\, \norm{\cont{}{t} - \cont{}{t-1}}\l \leqslant \Rcal_t}\). Observe that the set \(\mathcal{X} \subset \wcal\) is nonempty and compact, and thus invoking \cite[Theorem 2.2]{ref:OG10}, the first assertion follows.

Now we prove the second assertion that involves the conditions  \eqref{thrmii1}-\eqref{thrmii3}. We denote the cost function by
\begin{equation}
    J(\optpair_{T}) \Let \sum_{t=0}^{T-1} c(t,\stt{}{t},\cont{}{t}) + \cost_F(T,\stt{}{T}). 
\end{equation}
Let us define the sublevel sets of \(J\), by
\(
L_{\alpha}(J) \Let \aset[\big]{\optpair_T  \suchthat J(\optpair_T) \leqslant \alpha} 
\) for every \(\alpha \in \Rbb\). Observe that
\(L_{\alpha}(J)\) is closed since for all \(t=0,\ldots,T-1\), \(\cost(t,\cdot,\cdot)\) and \(\cost_{F}(T,\cdot)\) are lower semi-continuous, and \(L_{\alpha}(J)\) is bounded since \(J\) is weakly coercive. This implies that the sublevel set \(L_{\alpha}(J)\) is nonempty and compact. Invoking \cite[Theorem 2.3]{ref:OG10} immediately gives us the  second assertion. \hfill\(\qed\)


\begin{remark}
Observe that in Theorem \eqref{thrm:existence} it is possible to replace the condition \eqref{thrmii3} by the weaker condition \cite{ref:allaire2007numerical}:
\begin{equation}\label{eq:alter_cond}
    \inf_{\optpair_T \in \mathcal{X}}J(\optpair_T) < \lim_{R \to +\infty} \Bigl( \inf_{\norm{\optpair_T} \ge R}J(\optpair_T)\Bigr) < +\infty, \nn
\end{equation}
and then the assertion of Theorem \eqref{thrm:existence} still holds. 
\end{remark}

\begin{remark}
A general existence theorem for state-action constrained finite horizon discrete-time optimal control problem can be found in \cite{ref:JD-76}, where the state-action constraint sets are assumed to be Hausdorff topological spaces. For an infinite horizon case, we refer the readers to \cite{ref:SK:EG-83}.
\end{remark}

\begin{remark}
\label{rem:optimizers_1}
A translation of the constraints on the rate of the control actions into equivalent constraints on the control actions and/or states of the original system, in general, is not possible. The standard PMP in \cite{ref:VGB-75} cannot, therefore, be applied directly to the problem \eqref{eq:original problem}. There are at least two routes to attack the problem \eqref{eq:original problem}: One is to absorb the rates as the new \emph{action} variables and to derive a new discrete-time PMP. It has been treated extensively in \cite{ref:SG:SD:DC:RB-21}. The second approach is to absorb the rates as the new \emph{state} variables and to transform the original problem into a (possibly equivalent) new optimal control problem so that we can directly apply the standard PMP to the new one. We emphasize that the first approach does not provide us with first-order necessary conditions in the original state and control variables although it does produce a viable numerical scheme for synthesizing rate constrained optimal action trajectories. We propose an alternative formulation here in which the information about the rate constraints are captured by a sequence of extended states; see \S\ref{subsec:transformed problem} ahead. This procedure has the added advantage of giving us a rate constrained maximum principle directly with the original state and action variables.
    \end{remark}
\subsection{Transformed problem}\label{subsec:transformed problem}
The original problem presented in \S\eqref{sec:original problem} and all its associated data remain intact. We define a new set of variables to transcribe the rate constraints on the successive actions variables as new states, thus lifting the original problem to a higher dimensional state-space. To wit,  we define the sequence of new states \(\extst{t}{k} \in \Rbb^m\) where \(t = \timestamp{0}{T}\) and \(k = \timestamp{0}{T-2}\):
\begin{equation}
\label{eq:sys_new_st} 
    \begin{aligned}
        \begin{cases}
            \extst {0}{k}=0,\, \extst{1}{k}=\extst{0}{k},\ldots,\,\extst{k}{k}=\extst{k-1}{k};\\
            \extst{k+1}{k}=-\cont{}{k};\\
            \extst{k+2}{k}=\extst{k+1}{k}+\cont{}{k+1};\\
            \extst{k+3}{k}=\extst{k+2}{k},\, \extst{k+4}{k}=\extst{k+3}{k},\ldots,\, \extst{T}{k}=\extst{T-1}{k}.
        \end{cases}
    \end{aligned}
\end{equation}
The variables \(\extst{t}{k}\) captures the essence of the rate constraints at the \(k^{\text{th}}\) instant. For a fixed \(k \in \{0,\ldots,T-2\}\), we stipulate that \(\extst{t}{k} \in \ysetin{t}{k}\), where \(\ysetin{t}{k}\) brings in both the control and the rate constraints as:
 \begin{equation}
\begin{aligned}
\ysetin{t}{k} \Let\begin{cases}
  \{0\}~&\text{for}~t=\timestamp{0}{k},\\
\Ubb{t}~&\text{for} ~t=k+1,\\
\aset[\big]{\extst{t}{k} \in \Rbb^m \suchthat \|\extst{t}{k}\| \leqslant \Rcal_k} &\text{for}~t=\timestamp{k+2}{T}.
\end{cases}
\end{aligned}
\end{equation} 
To simplify our notation, we define set \begin{equation*}
    \ysetin{k+2}{k}\Let \aset[\big]{ \extst{k+2}{k} \in \Rbb^m \suchthat \|\extst{k+2}{k}\| \leqslant \Rcal_{k}}, \quad k \in \aset[]{0,\dots,T-2};
\end{equation*} 
note that \(\ysetin{k+2}{k}\) consists of the data of the \(k^{\text{th}}\) instant rate. Then, with this notation in place we have \( \ysetin{t}{k}=\ysetin{k+2}{k}\) for all \(t \in \{k+2,\ldots,T\}.\) Condensing the preceding definitions, we observe that the vector field corresponding to the non-autonomous dynamics describing the new set of states is given by
\begin{equation}
\label{eq:sys_rate_dyn} 
    \begin{aligned}
        \extst{t+1}{k}= \gt{k} \bigl(t,\extst{t}{k},\cont{}{t}\bigr)\Let\begin{cases}
        \extst{t}{k}~&\text{for all}~t \neq k,k+1,\\
        -\cont{}{t}~&\text{for}~t=k,\\
        \extst{t}{k}+\cont{}{t} &\text{for}~t=k+1,
        \end{cases}
    \end{aligned}
\end{equation}
where \( \bigl(\extst{t}{k}\bigr)_{t=0}^T \subset \Rbb^m\) for each \(k\). Thus, the original problem \eqref{eq:original problem} can now be written as
 \begin{align}
        \label{eq:transformed problem 1}
        \begin{aligned}
            &\minimize_{\substack{\stt{}{0},(\cont{}{t})_{t=0}^{T-1}}}  && \sum_{t=0}^{T-1}\cost\bigl(t,\stt{}{t},\cont{}{t}\bigr) +\cost_F(T,\stt{}{T})\\
            &\hspace{2mm}\sbjto && \begin{cases}
           \text{dynamics\,}\eqref{eq:system},\\
            \text{dynamics\,}\eqref{eq:sys_rate_dyn},\\
            \stt{}{t} \in \Mbb{t} \hspace{3mm}\text{for all \(t=\timestamp{0}{T}\)},\\
            \cont{}{t} \in \Ubb{t} \hspace{3.5mm}\text{for all \(t=\timestamp{0}{T-1}\)},\\
            \extst{t}{k} \in \ysetin{t}{k} \hspace{3mm}\text{for all}\,\,t=\timestamp{0}{T},\,\&\,\,k=0,\ldots,T-2,
            \end{cases}
        \end{aligned}
    \end{align}
which is in the standard form, as reported in \cite{ref:VGB-75}. To apply \cite[Theorem 20]{ref:VGB-75}, we define an extended set of states \(w\) and vector field \(\mathcal{F}\) concatenating the original state vectors \(\stt{}{t}\) along with the new state vectors \(\bigl(\extst{t}{k}\bigr)_{k=0}^{T-2}\) for each \(t=0,\ldots,T\), and the original vector field \((\dummyx,\dummyu)\mapsto\field\bigl(s,\dummyx,\dummyu\bigr)\) along with the vector field associated with the new state dynamics \eqref{eq:sys_rate_dyn}, by
\begin{equation}
\label{eq:w_t and F_t}
\wtt{}{t} \Let \begin{pmatrix}
     \stt{}{t}\\ \extst{t}{0} \\ \extst{t}{1}\\ \vdots\\ \extst{t}{T-2} 
\end{pmatrix} \quad\text{and}\quad
  \exdyn {\wtt{}{t}} {\cont{}{t}} \Let \begin{pmatrix}
       \field \bigl(t,\stt{}{t},\cont{}{t}\bigr) \\ \gt{0} \big( t,\extst{t}{0},\cont{}{t} \big)\\ \gt{1}\big( t,\extst{t}{1},\cont{}{t} \big)\\ \vdots \\ \gt{T-2}\big ( t,\extst{t}{T-2},\cont{}{t} \big) 
  \end{pmatrix}.
\end{equation}
The extended state vector \( \wtt{}{t} \) resides in \(\Rbb^{q}\) for every \(t = 0,\ldots,T\), where \(q \Let d+(T-1)m\), and the vector field \((\dummyw,\dummyu) \mapsto \mathcal{F}\bigl(s,\dummyw,\dummyu\bigr)\) denotes the extended set of dynamics for every \(s = 0,\ldots,T-1\). The optimal control problem \eqref{eq:transformed problem 1} is now transformed into
\begin{align}
        \label{eq:transformed problem 2}
        \begin{aligned}
            &\minimize_{\substack{\wtt{}{0},(\cont{}{t})_{t=0}^{T-1}}}  && \sum_{t=0}^{T-1}\cost\bigl(t,\stt{}{t},\cont{}{t}\bigr)+\cost_F(T,\stt{}{T})\\
            &\hspace{2mm}\sbjto && \begin{cases}
            \wtt{}{t+1}=\exdyn{\wtt{}{t}}{\cont{}{t}}\hspace{3mm}\text{for all}\,\,t=\timestamp{0}{T-1},\\
            \wtt{}{t} \in \wset{t} \hspace{4mm}\text{for all}\,\,t=\timestamp{0}{T},\\\cont{}{t} \in \Ubb{t} \hspace{5.5mm}\text{for all}\, \, t=\timestamp{0}{T-1},
            \end{cases}
        \end{aligned}
    \end{align}
    where the constraint set for the augmented state is defined as \[\wset{t} \Let \Mbb{t} \times \ysetin{t}{0} \times \ysetin{t}{1} \times \cdots \times \ysetin{t}{T-2}.\] Observe that the transformed problem \eqref{eq:transformed problem 2} has been lifted to a higher dimensional space where $\bigl(\wtt{}{t}\bigr)_{t=0}^{T} \in \Rbb^{q(T+1)}$ and $\bigl(\cont{}{t}\bigr)_{t=0}^{T-1} \in \Rbb^{mT}$. Moreover, the constraints in \eqref{eq:original problem} have been transformed into an equivalent set of state constraints in \eqref{eq:transformed problem 2} and the latter is in the standard form given in \cite{ref:VGB-75}. 
\subsection{Equivalence of the original and the transformed problems}\label{subsec:transformed prob}
Let us demonstrate that the original problem \eqref{eq:original problem} is equivalent to the transformed problem \eqref{eq:transformed problem 1} in view of following technical lemma \ref{lem:Equivalence}, the idea of which has been borrowed from \cite{ref:DM11}. This in turn, will establish the equivalence between \eqref{eq:original problem} and \eqref{eq:transformed problem 2}. To that end, we have the following lemma. 
\begin{lemma}\label{lem:Equivalence}
         Consider the following two optimization problems:
        \begin{align*}
            \begin{aligned}
                \p1: &\minimize_{x}  && f_1(x)\\
                &\sbjto && x \in \Sbb_1
           \end{aligned}
           & \qquad\text{and}\qquad
            \begin{aligned}
                \p2:& \minimize_{y}  && f_2(y)\\
                &\sbjto && y \in \Sbb_2 
           \end{aligned}
        \end{align*}
        where $\Sbb_1$ and $\Sbb_2$ denotes the feasible set corresponding to $\p1$ and $\p2$. Suppose that the objective functions $f_1(\cdot)$ and $f_2(\cdot)$ are continuous such that $\compose {f_2} F_{12}=f_1$ and $\compose {f_1} F_{21} = f_2$. Assume that a feasible point in $\p1 \,(\mathrm{or}\, \p2)$ is mapped to a feasible point in $\p2\,(\mathrm{or} \, \p1)$ under $F_{12}\,(\mathrm{or}\, F_{21})$. Under these hypothesis:
        \begin{enumerate}
            \item If $x \as$ is optimal for $\p1$, then $F_{12}(x \as)$ is optimal for $\p2$.
            \item If $y \as$ is optimal for $\p2$, then $F_{21}(y \as)$ is optimal for $\p1$.
        \end{enumerate}
    \end{lemma}
    \begin{proof}
        We start by assuming that $x \as$ is optimal for $\p1$. If $F_{12}(x \as)$ is not optimal in $\p2$, then there exists a feasible $\tilde{y}$ for which $f_2(\tilde{y})< f_2\big(F_{12}(x \as) \big).$ But then $F_{21}(\tilde{y})$ is feasible for $\p1$, and by assumption
        \begin{equation*}
            f_1\big(F_{21}(\tilde{y})\big)=f_2(\tilde{y}) < f_2 \big( F_{12}(x \as)\big)=f_1(x \as),
        \end{equation*}
        which contradicts optimality of $x \as$ in $\p1$. The converse directions admits an identical proof.  
    \end{proof}
    
\begin{remark}
Lemma \eqref{lem:Equivalence} can be parsed in terms of Figure \ref{fig:eqvt_diag}. The assertions in Lemma \ref{lem:Equivalence} follows if Figure \ref{fig:eqvt_diag} commutes. 
Note that $\mathcal{X}_1$, $\mathcal{X}_2$ can be manifolds of different dimensions.
\end{remark}
\begin{figure}[h]
\centering
\includegraphics[width=7cm,height=4cm]{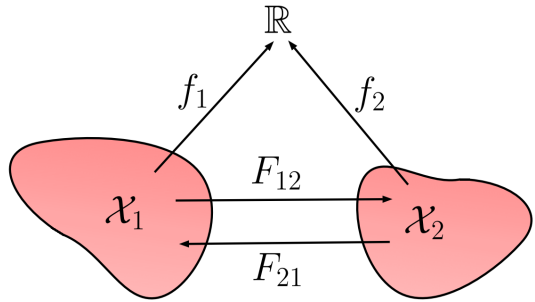}
\caption{Commutative diagram showing the equivalence between the optimization problems, \p1 and \p2.}
\label{fig:eqvt_diag}
\end{figure}
\noindent We establish an equivalence between \eqref{eq:original problem} and \eqref{eq:transformed problem 1}. We observe that for each \(k \in \{0,\ldots,T-2\}\), the equation \eqref{eq:sys_new_st} can be written in the matrix form
    \begin{equation*}
        \begin{pmatrix}
        I_m & 0_m & 0_m & \cdots & \cdots & \cdots & \cdots & \cdots & \cdots & 0_m\\
       -I_m & I_m & 0_m & 0_m & \cdots & \cdots  & \cdots & \cdots & \cdots & 0_m\\
        0_m & -I_m & I_m & 0_m & 0_m & \cdots  & \cdots & \cdots & \cdots& 0_m\\
        0_m & 0_m & -I_m & I_m & 0_m &  \cdots & \cdots & \cdots & \cdots& 0_m\\
        \vdots & \vdots & \ddots & \ddots & \ddots & \ddots & \ddots & \vdots & \vdots & \vdots\\
        0_m & \cdots & \cdots & 0_m & -I_m & I_m & 0_m & \cdots & \cdots & 0_m \\
        0_m & 0_m & \cdots & \cdots & \cdots & 0_m & I_m & 0_m & \cdots & 0_m\\
        0_m & 0_m & 0_m & \cdots & \cdots & \cdots & -I_m & I_m & \cdots & 0_m\\
        \vdots & \vdots & \vdots& \vdots & \vdots & \vdots &  \ddots & \ddots & \ddots & \vdots
      \\ 
       0_m & 0_m & 0_m & \cdots & \cdots & \cdots & \cdots & 0_m & -I_m & I_m
      \end{pmatrix}
      \begin{pmatrix}
           \extst{0}{k}\\\extst{1}{k}\\ \extst{2}{k} \\ \extst{3}{k} \\ \vdots\\ \extst{k}{k}\\ \extst{k+1}{k} \\ \extst{k+2}{k} \\  \vdots \\ \extst{T}{k}
      \end{pmatrix}=\begin{pmatrix}
       0_m \\ 0_m \\ 0_m \\ 0_m \\ \vdots \\ 0_m \\ -\cont{}{k} \\ \cont{}{k+1} \\ \vdots \\ 0_m    
      \end{pmatrix},
  \end{equation*}
  which we condense and write as \( A_k \rst_k = u_k\). The column vectors \(\rst_k\) and \(u_k\) reside in \(\Rbb^{m(T+1)} \). Observe that the matrix \(A_k\) is invertible for each \(k \in \aset[]{0,\ldots,T-2}\), which implies that the transformation matrix \(A \Let \mathrm{diag}\bigl(A_0,A_1,A_2, \ldots,A_k,\ldots, A_{T-2}\bigr)\) is in fact an invertible matrix. In view of Lemma \eqref{lem:Equivalence} if we define \( F_{12} \Let A\) and \(F_{21} \Let A^{\inverse}\), then \(F_{12} \circ F_{21} \) and \(F_{21}\circ F_{12}\) are the identity maps on \(\Rbb^{m(T+1)}\). This justifies the transformation done in \eqref{eq:sys_new_st} to show that the problems \eqref{eq:original problem} and \eqref{eq:transformed problem 1} are equivalent in the sense that if \(\Bigl( \bigl(\stt{\as}{t}\bigr)_{t=0}^{T},\bigl(\cont{\as}{t}\bigr)_{t=0}^{T-1}\Bigr)\) is an optimal trajectory for \eqref{eq:original problem}, then \(\Bigl( \bigl(\stt{\as}{t}\bigr)_{t=0}^{T},\bigl(y\as_k(t)\bigr)_{t=0}^{T},\bigl(\cont{\as}{t}\bigr)_{t=0}^{T-1}\Bigr)\) where \(k \in \aset[\big]{0,\ldots,T-2}\), is also an optimal trajectory for \eqref{eq:transformed problem 1}, and vice-versa.

\section{Main results}
\label{sec:main_result}
\begin{assumption}\label{assm:Assumption1}
    We stipulate the following
    \begin{itemize}
        \item The map \(\Rbb^d \times \Rbb^m \ni (\dummyx,\dummyu) \mapsto \field\bigl(s,\dummyx,\dummyu \bigr) \in \Rbb^d\) is continuously differentiable;
        \item the cost per stage function \(\Rbb^d \times \Rbb^m \ni (\dummyx,\dummyu) \mapsto \cost\bigl(s,\dummyx,\dummyu\bigr) \in \Rbb\) is continuously differentiable, and 
        \item the set \(\Mbb{t} \subset \Rbb^d\) is a nonempty and closed set for every \(t=0,\ldots,T\).
    \end{itemize}
\end{assumption}
\noindent The following theorem provides a set of first order necessary conditions for the optimal control problem \eqref{eq:original problem}, which is the main result of the article.
\begin{theorem}{\big(Rate constrained discrete-time PMP\big).}\label{rcpmp-p1}
Consider the optimal control problem \eqref{eq:original problem} with its associated data and suppose that \(\Bigl(\bigl(\stt{\as}{t}\bigr)_{t=0}^T,\bigl(\cont{\as}{t}\bigr)_{t=0}^{T-1}\Bigr)\) is an optimal state-action trajectory. Let assumption \ref{assm:Assumption1} hold. Define the Hamiltonian
\begin{equation}
    \label{eq:rcpmp_Ham-p1}
    \begin{aligned}
        \Rbb \times \Nz \times \bigl( \Rbb^d \bigr)\dual \times \bigl( \Rbb^{m(T-1)} \bigr)\dual \times \Rbb^d \times &\,\Rbb^m \ni \bigl( \psi_0,s,\psi,\lambda,\dummyx,\dummyu \bigr)\mapsto \\ H^{\psi_0}\bigl(s,\psi,\lambda,\dummyx,\dummyu \bigr) \Let  & \,\inprod{\psi}{f\bigl(s,\dummyx,\dummyu \bigr)}+\inprod{\lambda_{s-1}-\lambda_s}{\dummyu} - \psi_0 c \bigl(s,\dummyx,\dummyu \bigr),
    \end{aligned}
\end{equation}
where the sequence of vectors \(\lambda \Let \aset[]{\lambda_0,\ldots,\lambda_{T-2}} \subset \bigl(\Rbb^{m(T-1)}\bigr)\dual\). Then there exist 
\begin{enumerate}[leftmargin=*, widest=b, align=left]
    \item trajectories \(\bigl( \etaf(t) \bigr)_{t=0}^{T-1} \subset \bigl(\Rbb^{d}\bigr)^{\star}\) and \(( \etax(t)\bigr)_{t=0}^T \subset \bigl(\Rbb^d\bigr)^{\star}\),
    
    \item trajectories \( \etag{k}: \aset[]{0,\ldots,T-1} \lra (\Rbb^{m})^{\star}\) and the sequence \(\etay{k}:\aset[]{0,\ldots,T} \lra (\Rbb^m)^{\star}\) such that \(\etag{k}(T-1) =\etay{k}(T)\) and for all \(k=\timestamp{0}{T-2}\),  
\begin{equation}
\begin{aligned}
 \etag{k}(t-1) =\begin{cases}
  \etay{k}(t)~&~\text{for}~t=k,\\
\etag{k}(t)+\etay{k}(t)~&~\text{for all} ~t\neq k, \nn
\end{cases}
\end{aligned}
\end{equation}
\item $\abnormal \in \Rbb$,
\end{enumerate}
satisfying the following conditions 
	\begin{enumerate}[label=\textup{(\ref{rcpmp-p1}-\alph*)}, leftmargin=*, widest=b, align=left]
		\item \label{rcpmp:nonneg-p1} non-negativity condition:  $\abnormal \geqslant 0$;
		
		\item \label{rcpmp:non-triv-p1} nontriviality condition: the tuple \( \Bigl(\abnormal, \bigl( \etaf(t) \bigr)_{t=0}^{T-1}\Bigr)\) do not vanish simultaneously;
		
		\item \label{rcpmp:state-p1}  the system state dynamics:
			\begin{equation*}
				\begin{aligned}
				\stt{\as}{t+1} &=\frac{\partial }{\partial \psi}H^{\abnormal} \big(t,\etaf(t),\etag{0}(t),\ldots,\etag{T-2}(t),\stt{\as}{t}, \cont{\as}{t}  \big) \quad \text{for all}\,\, t=\timestamp{0}{T-1}\\
				\end{aligned}
			\end{equation*}
			
		\item \label{rcpmp:adj-p1}
		the system adjoint dynamics:
			\begin{equation*}
				\begin{aligned}
				\etaf(t-1)&=\frac{\partial }{\partial \dummyx}H^{\abnormal}\big(t,\etaf(t),\etag{0}(t),\ldots,\etag{T-2}(t),\stt{\as}{t},\cont{\as}{t}\big)\nn \\& \hspace{40mm}+\etax(t) \quad \text{for all}\,\, t=\timestamp{0}{T-1},
				\end{aligned}
			\end{equation*}
			
		\item \label{rcpmp:trv-p1} transversality conditions:
			\begin{equation*}
			    \begin{aligned}
			    \begin{cases}
			    \frac{\partial}{\partial \dummyx}H^{\abnormal}\big(0,\etaf(0),\etag{0}(0),\ldots,\etag{T-2}(0),\stt{\as}{0},\cont{\as}{0} \big)+\etax(0)=0, \\
			     \etaf(-1)=0,\quad \etax(T)-\etax(T-1)=0,\\ H^{\abnormal}\big(T,\etaf(T),\etag{0}(T),\ldots,\etag{T-2}(T),\stt{\as}{T},\cont{\as}{T}\big)=0,\end{cases}
			    \end{aligned}
			\end{equation*}
			
		\item \label{rcpmp:hmc-p1} the Hamiltonian maximization condition, pointwise in time: for every time\\ \(t=\timestamp{0}{T-1},\)
		\begin{equation}
		\begin{aligned}
			\Bigg\langle \frac{\partial}{\partial \dummyu}\biggl(\inprod{\etaf(t)}{\field\bigl(t,\stt{\as}{t},\cont{\as}{t} \bigr)}&+\inprod{\etag{t-1}(t)-\etag{t}(t)}{\cont{\as}{t}}\\& - \abnormal \cost \bigl(t,\stt{\as}{t},\cont{\as}{t} \bigr)\biggr),\tilde{u}(t) \Bigg \rangle \leqslant 0, 
			\end{aligned}
			\end{equation}
			whenever \(\cont{\as}{t}+\tilde{u}(t) \in U(t)\), where \(U(t)\) is a given local tent of \(\Ubb{t}\) at the point \(\cont{\as}{t}\).
	\end{enumerate}       
\end{theorem}
Observe that the necessary conditions in Theorem \eqref{rcpmp-p1} are in terms of the original variables, which is an advancement with respect to the results reported in \cite{ref:SG:SD:DC:RB-21}. We provide a complete proof of the theorem \eqref{rcpmp-p1} in Appendix \eqref{Appen:B}. 
\begin{remark}
A quintuple \(\Bigl(\abnormal,\bigl( \etaf(t) \bigr)_{t=0}^{T-1},\bigl(\etag{k} \bigr)_{t=0}^{T-1},\bigl(\stt{\as}{t}\bigr)_{t=0}^T,\bigl(\cont{\as}{t} \bigr)_{t=0}^{T-1}\Bigr)\) for all \(k=0,\ldots,T-2\), that satisfies the PMP in Theorem \ref{rcpmp-p1} is called an extremal lift of the optimal state-action trajectory, i.e., the tuple \(\Bigl(  \bigl(\stt{\as}{t}\bigr)_{t=0}^T,\bigl(\cont{\as}{t} \bigr)_{t=0}^{T-1} \Bigr)\). Extremal lifts with \(\abnormal=1\) are called normal extremals and the ones with \(\abnormal = 0\) are called abnormal extremals.
\end{remark}
\begin{remark}\label{rem:crucial}
Observe that the Hamiltonian maximization condition \ref{rcpmp:hmc-p1} in Theorem \ref{rcpmp-p1} does not induce an \emph{exact} maximum principle (see \cite[Chapter 6]{ref:BSM:06}) in the sense that the equality 
\begin{align}\label{e:exact_PMP}
	&H^{\abnormal}\bigl(t,\etaf(t),\etag{0}(t),\ldots,\etag{T-2}(t),\stt{\as}{t},\cont{\as}{t}\bigr) \nn \\&=\max_{\dummyu \in \mathbb{U}} 	H^{\abnormal}\bigl(t,\etaf(t),\etag{0}(t),\ldots,\etag{T-2}(t),\stt{\as}{t},\dummyu\bigr). 
\end{align}
does not hold.\footnote{Observe that \eqref{rcpmp:hmc-p1} signifies that the Hamiltonian \eqref{eq:rcpmp_Ham-p1} does not increase locally along directions that emanate from \(\cont{\as}{t}\) and enter the set \(\Ubb{t}\) for every \(t\). The Hamiltonian maximization condition \eqref{e:exact_PMP} cannot be reproduced at this level of generality \cite[Chapter 6]{ref:BSM:06}.} However, under the assumptions of convexity, compactness, and continuous dependence on \(\dummyx\) of the set 
\begin{equation}\label{e:vel_cond}
    V_t(\dummyx) \Let \aset[\big]{\bigl(\cost(t,\dummyx,\dummyu),\field(t,\dummyx,\dummyu)\bigr)\suchthat \dummyu \in \Ubb{t}} \quad 
    \text{where }\dummyx \in \Mbb{t},
\end{equation}
for \(t=0,\dots,T\), we recover the \emph{exact} Hamiltonian maximization condition \eqref{e:exact_PMP} as stated in \cite[Theorem 24]{ref:VGB-75}. For example, the exact Hamiltonian maximization condition \eqref{e:exact_PMP} holds for the control affine system with specific structural assumptions; see \S\ref{subsec:cas} and the Corollary  \ref{cor:cas} for more detail.
\end{remark}



\subsection{Special case of control affine systems}
\label{subsec:cas}
Consider a discrete-time autonomous control-affine system given by
\begin{align}
\label{eq:control-affine}
    \stt{}{t+1} &= \field \bigl(\stt{}{t}\bigr)+\agield \bigl(\stt{}{t} \bigr)\cont{}{t} \quad \text{for all}\,\,t=\timestamp{0}{T-1},
\end{align}
where \(\stt{}{t} \in \Rbb^{d}\) is the state vector and \(\cont{}{t} \in \Rbb^m\) is the control action at time \(t\). The family of maps \(\Rbb^d  \ni \dummyx \mapsto \field(\dummyx) \in \Rbb^d\) and \(\Rbb^d  \ni \dummyx \mapsto \agield(\dummyx) \in \Rbb^m\) are continuously differentiable, that constitute the dynamics corresponding to \eqref{eq:control-affine}.
Consider the optimal control problem
\begin{equation}
        \label{eq:control-affine-ocp}
        \begin{aligned}
            &\minimize_{\stt{}{0}, (\cont{}{t})_{t=0}^{T-1}}  && \sum_{t=0}^{T-1}\cost\bigl(\stt{}{t},\cont{}{t}\bigr) +\cost_F(\stt{}{T}) \\
            &\hspace{2mm}\sbjto && \begin{cases}
            \text{dynamics\,} \eqref{eq:control-affine},\\
            \stt{}{t} \in \mathbb{M} \hspace{3mm}\text{for all \,} \oTt,\\
            \cont{}{t} \in \mathbb{U} \hspace{3.5mm}\text{for all\,} \,\otzo,\\
            \norm{\cont{}{t+1} - \cont{}{t} } \leqslant \Rcal \hspace{3mm}\text{for all\,} \,\otzt,\\\cost\bigl(\dummyx,\cdot\bigr):\mathbb{U} \lra \Rbb \, \text{is convex for}\,\, \dummyx \in \Rbb^d, \, t=0,\ldots,T-1,\\
            \text{\(\mathbb{U} \neq \emptyset,\) convex, compact}\,\text{for every}\,\, t=0,\ldots,T-1.
            \end{cases}
        \end{aligned}
\end{equation}
\begin{corollary}\label{cor:cas}
(Rate-constrained discrete-time PMP for control-affine systems). Consider the optimal control problem \eqref{eq:control-affine-ocp} with its associated data and suppose that the state-action pair \(\Bigl(\bigl(\stt{\as}{t}\bigr)_{t=0}^T,\bigl(\cont{\as}{t}\bigr)_{t=0}^{T-1}\Bigr)\) is an optimal state-action trajectory. Let assumption \eqref{assm:Assumption1} hold. Moreover, assume that the cost per stage \(\cost(\dummyx,\cdot)\) is convex whenever \(\dummyx \in \Rbb^d\) and the constraint set \(\mathbb{U}\) is a nonempty, convex and compact subset of \(\Rbb^m\). Define the Hamiltonian
\begin{equation}
\label{eq:rcpmp_Ham-control-aff}
    \begin{aligned}
        H^{\abnormal}\Bigl(\etaf(t),\etag{0}(t),\ldots,\etag{T-2}(t),\stt{}{t},\cont{}{t} \Bigr)\Let
        & \inprod{\etaf(t)}{\field\bigl(\stt{}{t}\bigr)+ \agield\bigl(\stt{}{t}\bigr)\cont{}{t}}\\&+  \inprod{\etag{t-1}(t)-\etag{t}(t)}{\cont{}{t}}- \abnormal \cost \bigl(\stt{}{t},\cont{}{t}\bigr). \nn
    \end{aligned}
\end{equation}
Then there exist
\begin{enumerate}[leftmargin=*, widest=b, align=left]
    \item  trajectories \(\bigl( \etaf(t) \bigr)_{t=0}^{T-1} \subset \bigl(\Rbb^{d}\bigr)^{\star}\), the sequence \(\bigl( \etax(t)\bigr)_{t=0}^T \subset \bigl(\Rbb^d\bigr)^{\star}\),
    
    \item trajectories \( \bigl\{\etag{k}\bigr\}:\aset[]{0,\ldots,T-1} \lra \bigl(\Rbb^{m}\bigr)^{\star}\) and the sequence \(\bigl\{\etay{k}\bigr\}:\{0,\ldots,T-1\} \lra (\Rbb^m)^{\star}\), such that \(\etag{k}(T-1) =\etay{k}(T)\) for all \(k=0,1,\ldots,T-2\) and  
    \begin{equation*}
        \begin{aligned}
            \etag{k}(t-1) =\begin{cases}
            \etay{k}(t)~&~\text{for}~t=k,\\
            \etag{k}(t)+\etay{k}(t)~&~\text{for all} ~t\neq k,
            \end{cases}
        \end{aligned}
    \end{equation*}
    
    \item $\abnormal \in \Rbb$,
\end{enumerate}
satisfying the following conditions 
	\begin{enumerate}[label=\textup{(\ref{eq:control-affine-ocp}-\alph*)}, leftmargin=*, widest=b, align=left]
		\item \label{rcpmp:nonneg-con-aff} non-negativity condition:  $\abnormal \geqslant 0$;
		
		\item \label{rcpmp:non-triv-con-aff} nontriviality: the tuple \(\Big{(}\abnormal,\big{(}\etaud{\mathrm{f}}{t} \big{)}_{t=0}^{T-1}\Big{)}\) do not vanish simultaneously;
		
		\item \label{rcpmp:state-con-aff}  the system state dynamics: for all \(t=0,1,\ldots,T-1,\)
			\begin{equation*}
				\begin{aligned}
				\stt{\as}{t+1} &=\frac{\partial }{\partial \psi}H^{\abnormal} \bigl(\etaf(t),\etag{0}(t),\ldots,\etag{T-2}(t),\stt{\as}{t}, \cont{\as}{t}  \bigr),
				\end{aligned}
			\end{equation*}
			
		\item \label{rcpmp:adj-con-aff}
		 the system adjoint dynamics: for all \(t=0,1,\ldots,T-2,\)
			\begin{equation*}
				\begin{aligned}
				    \etaf(t-1)&=\frac{\partial }{\partial \dummyx}H^{\abnormal}\big(\etaf(t),\etag{0}(t),\ldots,\etag{T-2}(t),\stt{\as}{t},\cont{\as}{t}\big) +\etax(t),
				\end{aligned}
			\end{equation*}
			
		\item \label{rcpmp:trv-con-aff} transversality conditions:
			\begin{equation*}
			    \begin{aligned}
			    \begin{cases}\frac{\partial }{\partial \dummyx}H^{\abnormal}\big(\etaf(0),\etag{0}(0),\ldots,\etag{T-2}(0),\stt{\as}{0},\cont{\as}{0}\big)+\etax(0)=0 \\
			    \etaf(-1)=0,\quad \etax(T)-\etaf(T-1)=0,\\ H^{\abnormal}\big(\etaf(T),\etag{0}(T),\ldots,\etag{T-2}(T),\stt{\as}{T},\cont{\as}{T}\big)=0,\end{cases}
			    \end{aligned}
			\end{equation*}

		\item \label{rcpmp:hmc-con-aff} the Hamiltonian maximization condition, pointwise in time: For every\\time \(t=0,\ldots,T-1\)
	    \begin{equation*}
		    \begin{aligned}
		        &H^{\abnormal}\bigl(\etaf(t),\etag{0}(t),\ldots,\etag{T-2}(t),\stt{\as}{t},\cont{\as}{t}\bigr)\\ &=\max_{\dummyu \in \mathbb{U}} 	H^{\abnormal}\bigl(\etaf(t),\etag{0}(t),\ldots,\etag{T-2}(t),\stt{\as}{t},\dummyu \bigr),  
		    \end{aligned}
	    \end{equation*}
    \end{enumerate}    
\end{corollary}
The proof of this corollary is provided in Appendix \ref{Appen:B}.

\section{Discussion and numerical experiment}
\label{sec:numerical exp}
The necessary conditions for optimality \eqref{rcpmp:nonneg-p1}-\eqref{rcpmp:hmc-p1} yield a two point boundary value problem (TPBVP) that can be solved by indirect single shooting/multiple shooting (see \cite{ref:Rao-10,ref:Betts1998,ref:JZ:ET:MC-2017}). However, solving the rate constrained optimal control by indirect shooting is non-trivial in the sense that an extensive combinatorial search needs to be executed over the constraint state space for every time instant \(t = 0, \ldots,T\) and for every \(k = 0,\ldots, T-2\). This requires a thorough investigation of its own and will be reported in the subsequent articles. 

Nevertheless, against the preceding backdrop, we provide a \emph{contrived} numerical example to study the behavior of a discrete-time controlled recursion under the rate constraints. Let us consider the following linear discrete-time autonomous recursion
\begin{equation}
\label{num:sys}
    \begin{pmatrix}x_1(t+1)\\x_2(t+1)\\x_3(t+1)\end{pmatrix}=\begin{pmatrix}\cos(\pi/4)&-\sin(\pi/4)&0\\\sin(\pi/4)&\cos(\pi/4) & 0\\0&0&1 \end{pmatrix}\begin{pmatrix}
    x_1(t)\\x_2(t)\\x_3(t)\end{pmatrix}+\begin{pmatrix}
    0\\1\\1\end{pmatrix}u(t).
\end{equation}
Consider the linear quadratic optimal control problem, for \(T>0\) preassigned,
\begin{equation}
        \label{num:econtrol-affine-ocp}
        \begin{aligned}
            &\minimize_{(\cont{}{t})_{t=0}^{T-1} }  && \sum_{t=0}^{T-1}\frac{1}{2}\biggl(\inprod{x(t)}{Qx(t)}+\inprod{u(t)}{Ru(t)}\biggr)+\|x(T)\|^2\\
            &\sbjto && \begin{cases}
            \text{dynamics\,} \eqref{num:sys},\\
            \begin{cases}
            -8 \leqslant x_1(t) \leqslant 8 \hspace{3mm}\text{for all\,\,} \oTt,\\
            -8 \leqslant x_2(t) \leqslant 8 \hspace{3mm}\text{for all\,\,} \oTt,\\
            -0.2 \leqslant x_3(t) \leqslant 8 \hspace{3mm}\text{for all\,\,} \oTt,\\
            \end{cases}\\
            \|\cont{}{t}\| \leqslant 1 \hspace{3mm}\text{for all\,\,} \otzo,\\
            \norm{\cont{}{t+1} - \cont{}{t} } \leqslant 0.75 \hspace{3mm}\text{for all\,\,} \otzt,
            \end{cases}
        \end{aligned}
\end{equation}
where \(Q\) is a \(3 \times 3\) identity matrix, and \(R=0.5\). We present simulation results for \(T=30\). Figure \ref{num:fig_exmpl_states} demonstrates the state trajectories in response to the rate constrained control, while Figure \ref{num:fig_exmpl_control} shows the rate constrained control trajectory. We reformulate the above problem by introducing a new variable \(v(t) \Let u(t+1)-u(t)\) as in and translate the rate constraints as the new \emph{control} constraints. The problem \eqref{num:econtrol-affine-ocp} is thus now a standard discrete-time OCP with state-action constraints and we employ YALMIP \cite{ref:YALMIP_lofberg2004} to solve the problem via the quadratic programming solver \texttt{quadpprog} in MATLAB.
\par It can be seen that the successive increment of the control is restricted between \(-0.75\) and \(0.75\) (see Figure \ref{num:fig_exmpl_control}) for every time instant and hitting the constant level several times.

\begin{figure}[!ht]
\centering
\includegraphics[width=12cm,height=8cm]{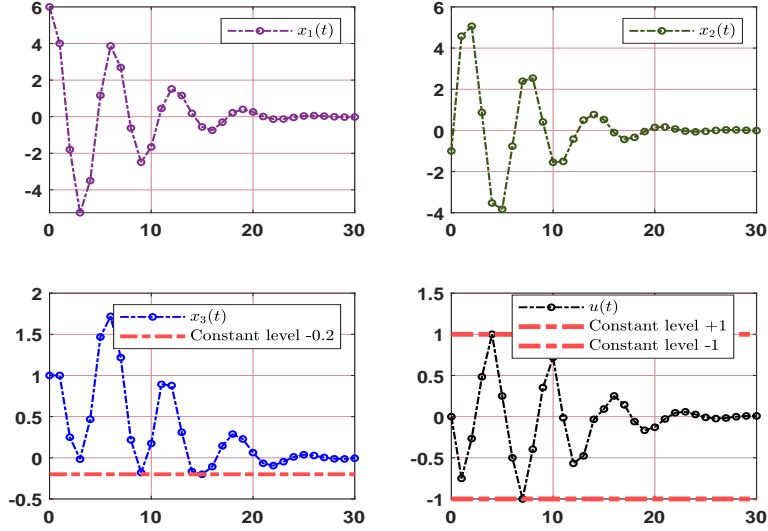}
\caption{Evolution of states and control trajectory corresponding to the system dynamics \eqref{num:sys}. The control trajectory \(\bigl(u(t)\bigr)_{t=0}^{T-1}\) (bottom-right) is obtained by solving the optimization problem \eqref{num:econtrol-affine-ocp}. The rate constraints are imposed during the design stage at every instant.}
\label{num:fig_exmpl_states}
\end{figure}

\begin{figure}[!ht]
\centering
\includegraphics[width=8.5cm,height=5cm]{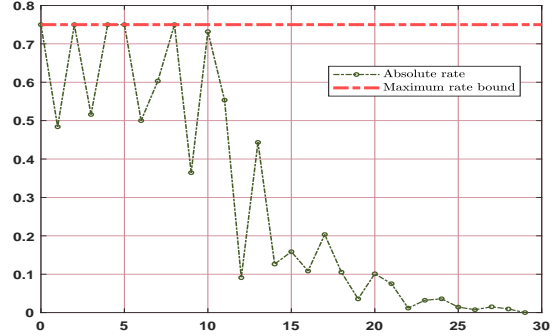}
\caption{Absolute value of the rate constraint imposed on the control trajectory. It can be observed that the rate constraint stays within the prescribed bound.} 
\label{num:fig_exmpl_control}
\end{figure}
Next, we demonstrate that if the constraints on the control action and its successive growth are not taken into account at the design stage, the consequences are drastic. This can be seen in Figure \ref{num:fig_nonex_1}. 
\begin{itemize}[leftmargin=*, widest=b, align=left]
	\item We consider the LQ problem \eqref{num:econtrol-affine-ocp} relaxing the constraints on the control action and its rate. 
    \item The optimization problem is solved and control trajectories are obtained.  Now we bound the control action and the rate, and feed it to the system dynamics \eqref{num:sys}. The corresponding response is shown in Figure \ref{num:fig_nonex_1}. This situation mimics a practical scenario where the constraints are inherently imposed by the actuators during the implementation stage. We can see from Figure \ref{num:fig_nonex_1} and \ref{num:fig_nonex_rate} that the generated control action is bounded between \(-1\) to \(1\) and is rate constrained between \(-0.75\) to \(0.75\). 
    \item Observe that the control trajectory obtained in Figure \ref{num:fig_nonex_1} is a rate constrained control trajectory that obeys the limitations imposed by the actuators, but choosing such a control action leads to an undesirable transient response, as is evident from Figure \ref{num:fig_nonex_1} and Figure \ref{num:fig_nonex_rate} and in many cases, is detrimental to the overall performance. 
	\end{itemize}
Hence, there exist a gap between the control actions that are being received at the actuator level and the ones that are faithfully executed.
\begin{remark}
    The numerical experiment in this section is provided to demonstrate the practical importance of rate constraints at the design stage. However, this article's objective is not to report developments on the numerical front using  indirect methods to solve the rate constrained optimal control problem \eqref{eq:original problem}. Such a development will be non-trivial and will require a separate investigation because a complicated search for the quantities \(\bigl(\etag{k}(t)\bigr)_{t=0}^{T-1}\) and \(\bigl(\etay{k}(t)\bigr)_{t=0}^T\) for every \(k=0,\ldots,T-2\), has to be performed in the dual cone contained in \(\bigl(\Rbb^m\bigr)\dual\), and will be reported in subsequent articles.
 \end{remark} 

\begin{figure}[!ht]
\centering
\includegraphics[width=12cm,height=8cm]{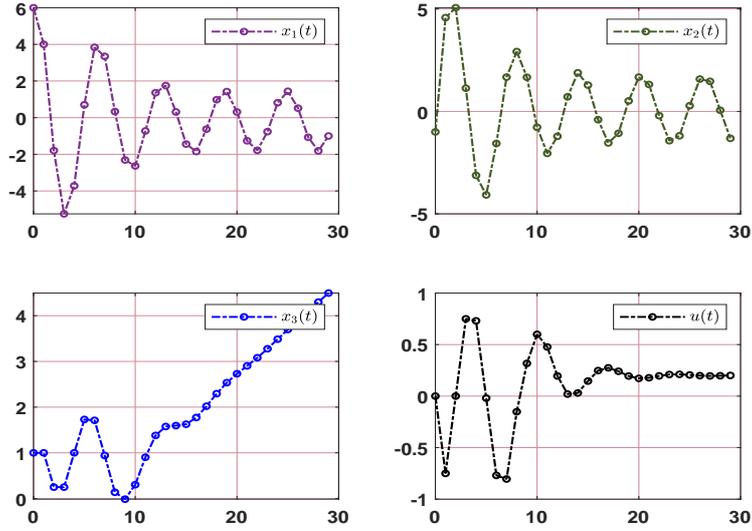}
\caption{Evolution of states trajectory corresponding to the state dynamics \eqref{num:sys}. The constraint on the control trajectories and their rate has not been taken into account during the design stage.} 
\label{num:fig_nonex_1}
\end{figure}

\begin{figure}[!htpb]
\centering
\includegraphics[width=8.5cm,height=5cm]{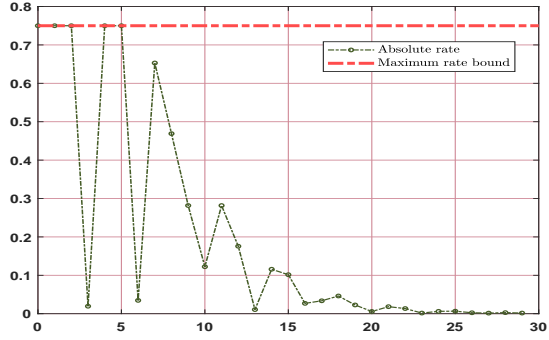}
\caption{Absolute value of the rate imposed on the control action during the implementation. The transient response of the system gets altered drastically, even though
the rate constraint stays within the prescribed bound \(r = 0.75\).}
\label{num:fig_nonex_rate}
\end{figure}


  
\section{Appendix A: Definitions}
\label{sec:appena}
This section presents some definitions verbatim and few basic theorems from the theory of convex analysis and tents. These standard definitions are lifted from \cite{ref:arxiv:PP:DC}.
\begin{itemize}[leftmargin=*, widest=b, align=left]
    \item Recall that a non-empty subset $K \subset \Rbb^{\genDim}$ is a cone if for every $y \in K$ and $\alpha \geqslant 0$ we have $\alpha y \in K$. In particular, $0 \in \Rbb^{\genDim}$ belongs to $K$. A non-empty subset $C \subset \Rbb^{\genDim}$ is convex if for every $y, y' \in C$ and $\theta \in [0, 1]$ we have $(1-\theta)y + \theta y' \in C$.
    \item A hyperplane $\Gamma$ in $\Rbb^{\genDim}$ is an $(\genDim -1)$-dimensional affine subset of $\Rbb^{\genDim}$. It can be viewed as the level set of a nontrivial linear function $p : \Rbb^{\genDim} \lra \Rbb$. If $p$ is given by $p(\genVar) = \inprod{a}{\genVar}$ for some $a (\neq 0) \in \Rbb^{\genDim}$, then
			\begin{equation*}
			     \Gamma \Let \big{\{}\genVar \in \Rbb^{\genDim} \mid \inprod{a}{\genVar} = \alpha\big{\}}.
			\end{equation*}
			\item We say that a family $\aset[]{\genTent_{0}, \genTent_{1}, \ldots, \genTent_{s}}$ of convex cones in $\Rbb^{\genDim}$ is separable if there exists a hyperplane $\Gamma$ and some $i \in \aset{0, \ldots, s}$ such that the cones $\genTent_{i}$ and $\bigcap_{j \neq i} \genTent_{j}$ are on two sides of $\Gamma$; formally, there exists $c \in \Rbb^{\genDim}$ and $i \in \aset{0, 1, \ldots, s}$ such that $K_i \subset \aset{y \in \Rbb^{\genDim} \suchthat \inprod{c}{y} \leqslant 0}$ and $\bigcap_{j \neq i} K_j \subset \aset{y\in\Rbb^{\genDim}\suchthat \inprod{c}{y} \geqslant 0}$. 
		\item Let $y \in \Rbb^{\genDim}$. A set $K \subset\Rbb^{\genDim}$ is a cone with vertex $y$ if it is expressible as $y + K'$ for some cone $K' \subset \Rbb^{\genDim}$. In particular, any cone is a cone with vertex $0 \in \Rbb^{\genDim}$.
	\item Let $\OMG$ be a nonempty set in $\Rbb^{\genDim}$. By $\affHull \OMG$ we denote the set of all affine combinations of points in $\OMG$. That is,
	\[
	\affHull \OMG = \aset[\Bigg]{\sum_{i=1}^{k} \theta_{i} \genVar_{i}\, \suchthat \sum_{i=1}^{k} \theta_{i} = 1, \quad \genVar_{i} \in \OMG \quad \text{for } i = 1, \ldots, k, \text{ and } k \in \N}.
    \]
	In other words, $\affHull \OMG$ is also the smallest affine set containing $\OMG$. The relative interior $\relInt \OMG$ of $\OMG$ denotes the interior of $\OMG$ relative to the affine space $\affHull \OMG$.
			\item Let $M$ be a convex set and $\vertex \in M$. The union of all the rays emanating from $\vertex$ and passing through points of $M$ other than $\vertex$ is a convex cone with vertex at $\vertex$. The closure of this cone is called the supporting cone of $M$ at $\vertex$.
			\item Let $\genTent \subset \Rbb^{\genDim}$ be a convex cone with vertex at $\vertex$. By $\dualCone{\genTent}$ we denote its polar (or dual) cone defined by
	\begin{equation*}
	\dualCone{\genTent} 
	\Let \aset[\big]{ y \in \big(\Rbb^{\genDim}\big){\dual} \, \big{|}\, \inprod{y}{x - \vertex} \leqslant 0 \quad \text{for all } x \in \genTent }.
	\end{equation*}
	It is clear that $\dualCone{\genTent}$ is a closed convex cone with vertex at $\vertex$ in view of the fact that it is an intersection of closed half-spaces:
	\[
	\dualCone{\genTent} = \bigcap_{y\in\genTent} \aset[\big]{z\in \big(\Rbb^{\genDim}\big){\dual}\,\big{|}\, \inprod{z}{y - \vertex} \geqslant 0}
	\]
\end{itemize}
\begin{theorem}[{{\cite[Theorem 7 on p.\ 10]{ref:VGB-75}}}]
\label{appen:bolt thrm 7 inseparability}
Let $s \in \N$, and for each $i = 1, \ldots, s$ let $\subSpace{i} \subset \Rbb^{\genDim}$ be a subspace satisfying $\subSpace{1} + \cdots + \subSpace{s} = \Rbb^{\genDim}$. For each $i = 1, \ldots, s$ let $\subSpace{i}^{\Delta}$ denote the direct sum of all subspaces $\subSpace{1}, \ldots, \subSpace{s}$  except $\subSpace{i}$, and $\genTent_{i}$ be a convex cone in $\subSpace{i}$ with a common vertex $\vertex \in \Rbb^{\genDim}$. If $N_{i} \Let \chull \bigl( \genTent_{i} \cup \subSpace{i}^{\Delta} \bigr)$ for each $i$, then $N_{i}$ is a convex cone, and the family $\aset[\big]{N_{i} \, \big{|}\, i = 1, \ldots s}$ is inseparable in $\Rbb^{\genDim}$, where we denote the convex hull of \(\genTent_{i} \cup \subSpace{i}^{\Delta}\) by \(\chull \bigl(\genTent_{i} \cup \subSpace{i}^{\Delta}\bigr).\)
\end{theorem}
Now move on to the theory of tents, where we will provide a few definitions and theorems. For a more comprehensive exposition on tents we refer \cite{ref:VGB-75}. 

\begin{definition}
		Let \(\OMG\) be a subset of $\Rbb^{\genDim}$ and let $\vertex \in \OMG$. A convex cone $Q \subset \Rbb^{\genDim}$ with vertex $\vertex$ is a tent of $\OMG$ at $\vertex$ if there exists a smooth map $\tentMap$ defined in a neighbourhood of $\vertex$ such that:\footnote{The theory also works for $\tentMap$ continuous.}
	\begin{enumerate}[leftmargin=*, align=left]
		\item $\tentMap (\genVar) = \genVar + o  (\genVar - \vertex )$,\footnote{\label{fn:o-Notation} Recall the Landau notation $\genFun (\genVar) = o (\genVar)$ that stands for a function $\genFun(0) = 0$ and $\lim_{\genVar \to 0} \frac{\abs{\genFun(\genVar)}}{\abs{\genVar}} = 0$.} and
		\item there exists $\veps > 0$ such that $\tentMap (\genVar) \in \OMG$ for $\genVar \in Q \cap \ball{\vertex}$.
	\end{enumerate}
	\end{definition}
	
We say that a convex cone $\genTent \subset \Rbb^{\genDim}$ with vertex at $\vertex$ is a local tent of $\OMG$ at $\vertex$ if, for every $\genVar \in \relInt \genTent$, there is a convex cone $Q \subset \genTent$ with vertex at $\vertex$ such that $Q$ is a tent of $\OMG$ at $\vertex$, $\genVar \in \relInt Q$, and $\affHull Q = \affHull \genTent$. Observe that if $\genTent$ is a tent of $\OMG$ at $\vertex$, then $\genTent$ is a local tent of $\OMG$ at $\vertex$.

We need the following theorems on tents in the formulation of our PMP in the sequel.

	\begin{theorem}[{{\cite[Theorem 8 on p.\ 11]{ref:VGB-75}}}]
	\label{appen:bolt thrm 8 tangent plane}
		Let $\OMG$ be a smooth manifold in $\Rbb^{\genDim}$ and $\genTent$ the tangent plane to $\OMG$ at $\vertex \in \OMG$. Then $\genTent$ is a tent of $\OMG$ at $\vertex$.
	\end{theorem}

	\begin{theorem}[{{\cite[Theorem 9 on p.\ 12]{ref:VGB-75}}}]
	\label{appen:bolt thrm 9 half-space tent }
		Given a smooth function $\genFun : \Rbb^{\genDim} \lra \Rbb$, let $\vertex$ be such that $\derivative{\genVar}{\genFun (\vertex )} \neq 0$. Define sets $\OMG, \OMG_{0} \in \Rbb^{\genDim}$ as
		\[
			\OMG \Let  \aset[\big]{ \genVar \in \Rbb^{\genDim}\, \big{|}\, \genFun (\genVar ) \leqslant \genFun (\vertex )}, \quad \OMG_{0} \Let \aset[\big]{ \vertex } \cup  \left\{\genVar \in \Rbb^{\genDim} \, \big{|}\, \genFun (\genVar ) < \genFun (\vertex )\right\}.
		\]
		Then the half-space $\genTent$ given by the inequality $\inprod{\derivative{\genVar}{\genFun (\vertex )}}{\genVar - \vertex} \leqslant 0$ is a tent of both $\OMG$ and $\OMG_{0}$ at $\vertex$.
	\end{theorem}

	\begin{theorem}[{{\cite[Theorem 10 on p.\ 12]{ref:VGB-75}}}]
	\label{appen:bolt thrm 10 supporting cone tent}
		Let $\OMG \in \Rbb^{\genDim}$ be a convex set and let $\genTent$ be its supporting cone at $\vertex \in \OMG$. Then $\genTent$ is a local tent of $\OMG$ at $\vertex$.
	\end{theorem}

	\begin{theorem}[{{\cite[Theorem 12 on p.\ 14]{ref:VGB-75}}}]
	\label{appen:bolt thrm 12 separability}
		Let $\OMG_{0}, \OMG_{1}, \ldots, \OMG_{s}$ be  subsets of $\Rbb^{\genDim}$ with a common point $\vertex$, and $\genTent_{0}, \genTent_{1}, \ldots, \genTent_{s}$ local tents of these sets at $\vertex$. If the family of cones $\aset[\big]{\genTent_{0}, \genTent_{1}, \ldots, \genTent_{s}}$ is inseparable and at least one of the cones is not a plane, then there exists $\genVar' \in \OMG_{0} \cap \OMG_{1} \cap \ldots \cap \OMG_{s}$ and $\genVar' \neq \vertex$.
	\end{theorem}

	\begin{proposition}
	\label{appen:prop optimality condition}
		A function $\genFun (\genVar )$ considered on the set $\feas = \OMG_{1} \cap \ldots \cap \OMG_{s}$, attains its minimum at $\vertex$ if and only if
		\[
			\Omega_{0} \cap \OMG_{1} \cap \ldots \cap \OMG_{s} = \{\vertex\},
		\]
		where $\Omega_{0} \Let \{\vertex \} \cup \aset[\big]{\genVar \in \Rbb^{\genDim} \, \big{|}\, \genFun (\genVar) < \genFun (\vertex)}$.
	\end{proposition}

\begin{theorem}[{{\cite[Theorem 16 on p.\ 20]{ref:VGB-75}}}]
	\label{appen:bolt thrm 16 necessary condition}
	Let $\OMG_{1}, \ldots, \OMG_{s}$ be subsets of $\Rbb^{\genDim}$ and let $\bigcap_{k=1}^s \OMG_k\ni\genVar\mapsto \genFun (\genVar )\in\Rbb$ be a smooth function. Let $\feas = \bigcap_{k=1}^s\OMG_{k}$, let $\vertex \in \feas$, and let $\genTent_{i}$ be a local tent of $\OMG_{i}$ at $\vertex$ for $i = 1, \ldots, s$. If $\genFun$ attains its minimum relative to $\feas$ at $\vertex$, then there exist vectors $\lambda_{i} \in \dualCone{\genTent_{i}}$ for $i = 1, \ldots, s$ and $\psi_{0} \in \Rbb$ satisfying
		\[
			-\psi_{0} \derivative{\genVar}{\genFun (\vertex )} + \lambda_{1} + \cdots + \lambda_{s} = 0
		\]
		such that $\psi_{0} \geqslant 0$, and if $\psi_{0} = 0$, then at least one of the vectors $\lambda_{1}, \ldots, \lambda_{s}$ is not zero.
	\end{theorem}
	
 \section{Appendix: Proofs}\label{Appen:B}
We provide a detailed proof of Theorem \eqref{rcpmp-p1} and  Corollary \ref{cor:cas} which is the main result of this article. \newline  \\
\noindent
{\it \textbf{Proof of Theorem \eqref{rcpmp-p1}:}} 
Consider the optimal control problem \eqref{eq:transformed problem 2}. Define the Hamiltonian
\begin{align}\label{eq:rcpmp-ham-p2}
    \Rbb \times \Nz \times (\Rbb^{\nu})^{\star} \times \Rbb^d \times \Rbb^{q} & \times\Rbb^m \ni (\psi_0,s,\psi,\dummyx,\dummyw,\dummyu) \mapsto \nonumber \\ &H^{\psi_0}(s,\psi,\dummyx,\dummyw,\dummyu) \Let  \inprod{\psi}{\mathcal{F}\bigl(s,\dummyw,\dummyu\bigr)} - \psi_0 \cost(s,\dummyx,\dummyu) \in \Rbb.
\end{align}
Observe that the problem \eqref{eq:transformed problem 2} is in the standard form as in \cite{ref:VGB-75}. Let \(\Bigl(\bigl(\wtt{\as}{t}\bigr)_{t=0}^T,\bigl(\cont{\as}{t}\bigr)_{t=0}^{T-1}\Bigr)\) be an optimal state-action trajectory of \eqref{eq:transformed problem 2}. Invoking \cite[Theorem 20]{ref:VGB-75}, there exists
 \renewcommand{\theenumi}{\Roman{enumi}}
\begin{enumerate}[leftmargin=*, widest=b, align=left]
    \item  trajectories \(\bigl( \etaF(t) \bigr)_{t=0}^{T-1} \subset \bigl(\Rbb^{q}\bigr)^{\star}\),
    \item sequences \(\bigl( \etaw(t) \bigr)_{t=0}^{T} \subset \bigl(\Rbb^{q}\bigr)^{\star}\) and, 
    \item $\abnormal \in \Rbb$,
\end{enumerate}
satisfying the following conditions:
	\begin{enumerate}[label=\textup{(PMP-\alph*)}, leftmargin=*, widest=b, align=left]
		\item \label{rcpmp:nonneg-p2} non-positivity condition:  \(\abnormal \geqslant 0\);
		\item \label{rcpmp:non-triv-p2} nontriviality: the tuple \( \Bigl( \abnormal,\bigl(\etaF(t)\bigr)_{t=0}^{T-1},\bigl(\etaw(t)\bigr)_{t=0}^{T}\Bigr)\) do not vanish simultaneously;
		\item \label{rcpmp:state-p2}  the system state dynamics:
			\begin{equation*}
				\begin{aligned}
				\wtt{\as}{t+1}&=\frac{\partial }{\partial \psi}H^{\abnormal} \big(t,\etaF(t),\wtt{\as}{t},\cont {\as}{t} \big)\\
				&= \mathcal{F}\bigl(t,\wtt{\as}{t},\cont{\as}{t}\bigr) \quad\text{for all}\,\, t=0,\ldots,T-1,
				\end{aligned}
			\end{equation*}
		\item \label{rcpmp:adj-p2}
		the system adjoint dynamics:
			\begin{equation*}
				\begin{aligned}
				\etaF(t-1)&=\frac{\partial }{\partial \dummyw}H^{\abnormal}\bigl(t,\etaF(t),\wtt{\as}{t},\cont {\as}{t}\big)+\etaw(t) \quad \text{for all}\,\, t=0,\ldots,T,
				\end{aligned}
			\end{equation*}
		\item \label{rcpmp:trv-p2} the transversality conditions:
			\begin{equation*}
			    \begin{aligned}
			    &\frac{\partial}{\partial \dummyw}H^{\abnormal}\big(0,\etaF(0),\wtt{\as}{0},\cont{\as}{0} \big)+\etaw(0)=0, \\
			    &\etaF(-1)=0,\quad\etaw(T)-\etaF(T-1)=0,\\ &H^{\abnormal}\big(T,\etaF(T),\wtt{\as}{T},\cont{\as}{T} \big)=0,
			    \end{aligned}
			\end{equation*}
			along with the boundary conditions: \(y\as_k(0)=0\,\text{for every}\,\, k=0,\ldots,T-2\),  
\begin{align}
    \text{and,\, }\begin{cases}
         y\as_0(T)=\cont{\as}{1}-\cont{\as}{0},\\
         y\as_1(T)=\cont{\as}{2}-\cont{\as}{1},\\ \hspace{3mm}\vdots \\
         y\as_{T-2}(T)=\cont{\as}{T-1}-\cont{\as}{T-2},
    \end{cases}
\end{align}
		  
		\item \label{rcpmp:hmc-p2} the Hamiltonian maximization condition, pointwise in time: for every \(t=0,\ldots,T-1\), we have	
			\[
			\Bigg\langle \frac{\partial}{\partial \dummyu}H^{\abnormal}\big(t,\etaf(t),\etag{0}(t),\ldots,\etag{T-2}(t),\stt{\as}{t},y \as_0(t),\ldots,y \as_{T-2}(t),\cont{\as}{t} \big),\tilde{u}(t)\Bigg \rangle \leqslant 0, 
			\]
			whenever \(\cont{\as}{t}+ \tilde{u}(t) \in U(t)\), where \(U(t)\) is a given local tent of \(\Ubb{t}\) at the point \(\cont{\as}{t}\).
	\end{enumerate}       
Observe that the vectors \(\bigl(\etaF(t)\bigr)_{t=0}^{T}\) denotes the co-state trajectories of the augmented system \eqref{eq:w_t and F_t}, and the individual components in the vector \(\etaF(t)\) denote the co-states  corresponding to the dynamics of \eqref{eq:system} and \eqref{eq:sys_rate_dyn}. Similar arguments hold for the vector \(\etaw(t)\) as well. Precisely, 
\begin{equation}
    \etaF(t)= \begin{pmatrix}
         \etaf(t)\\ \etag{0}(t)\\ \vdots\\ \etag{T-2}(t)
    \end{pmatrix} \quad \text{and} \quad \etaw(t) = \begin{pmatrix}
         \etax(t)\\ \etay{0}(t)\\ \vdots \\ \etay{T-2}(t)
    \end{pmatrix}.
\end{equation}
Then we can write the Hamiltonian as
\begin{multline}
    H^{\abnormal} \bigl(t,\etaF(t),\wtt{}{t},\cont{}{t} \bigr)=H^{\abnormal}\bigl(t,\etaf(t),\etag{0}(t),\ldots,\etag{T-2}(t),\stt{}{t},\extst{t}{0},\ldots,\extst{t}{T-2},\cont{}{t} \bigr) \nn \\= \inprod{\etaf(t)}{\field \bigl(t,\stt{}{t},\cont{}{t}\bigr)} \nn + \sum_{k=0}^{T-2}\inprod{\etag{k}(t)}{g_k\bigl(t,\extst{t}{k},\cont{}{t}\bigr)} - \abnormal \cost \bigl(t,\stt{}{t},\cont{}{t}\bigr) .
\end{multline}
Define \(H_{k}^{\abnormal}\bigl(t,\etaF(t),\wtt{\as}{t},\cont {\as}{t} \bigr) \Let \inprod{\etag{k}(t)}{g_{k}\bigl(t,y^{\ast}_k(t),\cont{\as}{t}\bigr)}\) for \(t =\timestamp{0}{T-1}\), and \(k = 0,\ldots,T-2\). Then we have 
\begin{equation}
   \wtt{\as}{t+1}=\begin{pmatrix}
     \stt{\as}{t+1}\\y_0^{\ast}(t+1)\\\vdots\\ y_{T-2}^{\ast}(t+1)
\end{pmatrix}=\begin{pmatrix}
     \frac{\partial}{\partial \psi_{\st}}H^{\abnormal}\bigl(t,\etaF(t),\wtt{\as}{t},\cont{\as}{t}\bigr) \\ \frac{\partial}{\partial \psi_{\rst}}H_0^{\abnormal}\bigl(t,\etaF(t),\wtt{\as}{t},\cont{\as}{t}\bigr) \\ \vdots \\ \frac{\partial}{\partial \psi_{\rst}}H_{T-2}^{\abnormal}\bigl(t,\etaF(t),\wtt{\as}{t},\cont{\as}{t}\bigr)
\end{pmatrix}, 
\end{equation}
giving us the chain of adjoint variables
  \begin{align}
    \etaf(t-1)&=\frac{\partial}{\partial \dummyw_{\st}}H^{\abnormal}\bigl(t,\etaF(t),\wtt{\as}{t},\cont{\as}{t}\bigr)+\etax(t), \nn \\ \etag{0}(t-1)&=\frac{\partial}{\partial \dummyw_{\rst}}H^{\abnormal}_0\bigl(t,\etaF(t),\wtt{\as}{t},\cont{\as}{t}\bigr)+\etay{0}(t), \nn \\ & \hspace{1.5mm}\vdots \nn\\ \etag{T-2}(t-1)&= \frac{\partial}{\partial \dummyw_{\rst}}H_{T-2}^{\abnormal}\bigl(t,\etaF(t),\wtt{\as}{t},\cont{\as}{t}\bigr)+\etay{T-2}(t).
\end{align}
We can simplify the adjoint dynamics and write
\begin{align}
\label{eq:eta_g_t-1_k}
    \etag{k}(t-1)=\frac{\partial}{\partial \dummyw_{\rst}}H_{k}^{\abnormal}\bigl(t,\etaF(t),\wtt{\as}{t},\cont{\as}{t}\bigr)+\etay{k}(t), 
\end{align}
where \(t=0,1,\ldots,T\). In general, for all \(k=0,1,\ldots,T-2\), \eqref{eq:eta_g_t-1_k} can be written as  \( \etag{k}(T-1)=\etay{k}(T)\) and 
\begin{equation}
\label{eq:eta_g_seq}
        \begin{aligned}
            \etag{k}(t-1) =\begin{cases}
            \etay{k}(t)~&~\text{for}~t=k,\\
            \etag{k}(t)+\etay{k}(t)~&~\text{for all} ~t\neq k.
            \end{cases}
        \end{aligned}
    \end{equation}
Let us show for one instance, how to arrive at \eqref{eq:eta_g_seq}. Fix \(k=0\), let \(t=0,\ldots,T\), then from \eqref{eq:eta_g_t-1_k} and \eqref{eq:sys_rate_dyn} we get
\begin{equation}
    \begin{alignedat}{2}
        \etag{0}(-1)&=\frac{\partial}{\partial \dummyw_{\rst}} \inprod{\etag{0}(0)}{-u\as(0)}+\etay{0}(0)=\etay{0}(0) \quad&\text{for \(t=0,\)}\nn \\\etag{0}(0)&=\frac{\partial}{\partial \dummyw_{\rst}} \inprod{\etag{0}(1)}{y_0^{\ast}(1)+\cont{\as}{1}}+\etay{0}(1)=\etag{0}(1)+\etay{0}(1) \quad&\text{for \(t=1,\)} \nn \\ \etag{0}(1)&=\frac{\partial}{\partial \dummyw_{\rst}} \inprod{\etag{0}(2)}{y_0^{\ast}(2)}+\etay{0}(2)=\etag{0}(2)+\etay{0}(2) \quad&\text{for \(t=2,\)} \nn \\&\hspace{1.5mm}\vdots \nn \\ \etag{0}(T-1)&=\frac{\partial}{\partial \dummyw_{\rst}} \inprod{\etag{0}(T)}{y_0^{\ast}(T)}+\etay{0}(T)=\etay{0}(T) \quad&\text{for \(t=T,\)}
    \end{alignedat}
\end{equation}
where the last condition at \(t=T\) is a specific instance of transversality condition \eqref{rcpmp:trv-p2}. Thus, going through this recursive procedure, we can see that \eqref{eq:eta_g_seq} holds. From the Hamiltonian maximization condition for every \(t=\timestamp{0   }{T-1} \),		
			\[
			\Bigg\langle \frac{\partial}{\partial \dummyu}H^{\abnormal}\big(t,\etaf(t),\etag{0}(t),\ldots,\etag{T-2}(t),\stt{\as}{t},y \as_0(t),\ldots,y \as_{T-2}(t),\cont{\as}{t}\big),\tilde{u}(t) \Bigg \rangle \leqslant 0, 
			\]
			whenever \(\cont{\as}{t}+ \tilde{u}(t) \in U(t)\), where \(U(t)\) is a given local tent of \(\Ubb{t}\) at the point \(\cont{\as}{t}\).
Observe that for all \(t=\timestamp{0}{T-1}\),
\begin{alignat}{2}
\label{rcpmp:hamilt_eq}
&\Bigg\langle \frac{\partial}{\partial\dummyu}H^{\abnormal}\bigl(t,\etaF(t),\wtt{\as}{t},\cont{\as}{t} \bigr),\tilde{u}(t) \bigg \rangle \leqslant 0 \nn \\ 
\iff& \bigg \langle \frac{\partial}{\partial\dummyu}\bigg{(} \inprod{\etaf(t)}{\field\bigl(t,\stt{\as}{t},\cont{\as}{t}\bigr)} +\sum_{k=0}^{T-2}\inprod{\etag{k}(t)}{g_{k}\bigl(t,y \as_k(t),\cont{\as}{t}\bigr)} \\& \hspace{20mm} -\abnormal \cost\bigl(t,\stt{\as}{t},\cont{\as}{t}\bigr)\bigg{)}   ,\tilde{u}(t)\bigg\rangle \leqslant 0 \nn \\
\iff& \Bigg\langle \frac{\partial}{\partial\dummyu} \biggl(  \inprod{\etaf(t)}{\field\bigl(t,\stt{\as}{t},\cont{\as}{t}\bigr)} + \inprod{\etag{t-1}(t)-\etag{t}(t)}{\cont{\as}{t}} \\& \hspace{20mm} - \abnormal \cost\bigl(t,\stt{\as}{t},\cont{\as}{t}\biggr), \tilde{u}(t)\bigr) \bigg\rangle \leqslant 0 , 
\end{alignat}
 whenever \(\cont{\as}{t}+\tilde{u}(t) \in U(t)\), where \(U(t)\) is a given local tent of \(\Ubb{t}\) at \(\cont{\as}{t}\).
 If we define the Hamiltonian for the rate constrained optimal control problem \eqref{eq:original problem} by 
 \begin{equation}
 \label{eq:HHH}
    \begin{aligned}
         H\bigl(t, \etaf(t),&\etag{0}(t),\etag{1}(t),\ldots,\etag{T-2}(t),\stt{\as}{t},\cont{\as}{t} \bigr) \Let  \inprod{\etaf(t)}{\field\bigl(t,\stt{\as}{t},\cont{\as}{t}\bigr)}\\&+ \inprod{\etag{t-1}(t)-\etag{t}(t)}{\cont{\as}{t}}-\abnormal \cost\bigl(t,\stt{\as}{t},\cont{\as}{t}\bigr),
     \end{aligned}
 \end{equation}
then we have the following conclusions:
\begin{enumerate}[leftmargin=*, widest=b, align=right]
    \item[(i)] The non-positivity and  non-triviality conditions \eqref{rcpmp:nonneg-p1}-\eqref{rcpmp:non-triv-p1} can be obtained from \eqref{rcpmp:nonneg-p2}-\eqref{rcpmp:non-triv-p2}.
    \item[(ii)] Defining the Hamiltonian as in \eqref{eq:HHH}, the system dynamics \eqref{rcpmp:state-p2} and the adjoint dynamics \eqref{rcpmp:adj-p2} are seen to be identical to the one in \eqref{rcpmp:state-p1}.
    \item[(iii)] The transversality condition \eqref{rcpmp:trv-p1} can be obtained from the Hamiltonian \eqref{eq:HHH}, \eqref{rcpmp:trv-p2} and from the definition of the trajectory \(\bigl(\etag{k}(t)\bigr)_{t=0}^{T-1}\) and the sequence \(\bigl(\etay{k}(t)\bigr)_{t=0}^{T}\).
    \item[(iv)]We observe that from \eqref{rcpmp:hamilt_eq} the Hamiltonian maximization condition in \eqref{rcpmp:hmc-p2} is equivalent to the Hamiltonian maximization condition \eqref{rcpmp:hmc-p1}.
\end{enumerate}
Hence, we can deduce the conditions \eqref{rcpmp:nonneg-p1}-\eqref{rcpmp:hmc-p1} of Theorem \ref{rcpmp-p1} from the conditions \eqref{rcpmp:nonneg-p2}-\eqref{rcpmp:hmc-p2}. This completes the proof. \hfill \(\qed\)

{\it \textbf{Proof of Corollary \eqref{cor:cas}}}: The conditions \eqref{rcpmp:nonneg-con-aff}--\eqref{rcpmp:trv-con-aff} follows from the conditions \eqref{rcpmp:nonneg-p1}--\eqref{rcpmp:trv-p1} in Theorem \ref{rcpmp-p1}. It remains to show that \eqref{rcpmp:hmc-con-aff} holds: note that \(H^{\abnormal}\bigl(\etaf(t),\etag{0}(t),\ldots,\etag{T-2}(t),\stt{\as}{t},\cdot\bigr): \Ubb{t} \lra \Rbb\) is concave. Indeed, from convexity of \(\cost(\dummyx,\cdot)\) it follows that  
    \begin{align}
       &\left( \frac{\partial^2}{\partial\dummyu^i \partial\dummyu^j}H^{\abnormal}\bigl(\etaf(t),\etag{0}(t),\ldots,\etag{T-2}(t),\stt{\as}{t},\cont{\as}{t}\bigr)\right)_{i,j} \nn\\&= -\left(\frac{\partial^2}{\partial\dummyu^i \partial\dummyu^j}\cost(\stt{\as}{t},\cont{\as}{t}) \right)_{i,j}  \preccurlyeq 0. \nn
    \end{align}
Since ,\(\mathbb{U}\) is compact, invoking Weierstrass Theorem \cite[Theorem 2.2]{ref:OG10}, the maximization condition \eqref{rcpmp:hmc-con-aff} now follows immediately. \hfill \(\qed\)

\bibliographystyle{amsalpha}
\bibliography{refs}

\end{document}